\documentclass[11pt]{amsart}

\usepackage{amsmath}
\usepackage{amssymb}
\usepackage{graphicx}

\hoffset = - 0.5in

\newtheorem{theorem}{Theorem}[section]
\newtheorem{corollary}[theorem]{Corollary}

\theoremstyle {definition}

\numberwithin{equation}{section}

\renewcommand{\geq}{\geqslant}
\renewcommand{\leq}{\leqslant}

\title{Uniformly nonsquare Banach spaces have the fixed point property 2}

\author[Tim Dalby]{Tim Dalby}

\date{\today}



\email{tim\_dalby@bigpond.com}

\begin{document}

\parindent = 0pt
\parskip = 8pt

\begin{abstract}

Another proof that uniformly nonsquare Banach spaces have the fixed point property is presented.

\end{abstract}

\maketitle

\section{Introduction}

In [5] Garc\'{i}a-Falset, Llorens-Fuster and Mazcu\~{n}an-Navarroa were the first to show that a uniformly nonsquare Banach space has the Fixed Point Property, FPP.  The proof is not direct because it travels via the modulus of smoothness to the modulus of nearly uniform smoothness to the coefficient $RW(a, X).$  For more information about $RW(a, X)$ see [1] or [2].

There is now a very direct proof of uniformly nonsquare implying the FPP courtesy of Dowling, Randrianantoanina and Turett, [3].

In [4] $J(X)$ was defined and used in the discussion about the unit sphere in some well known Banach spaces.  $J(X)$ is known as the James constant and is also called the Gau-Lau coefficient, $G(X)$. It is defined by
\[ J(X) = \sup \{ \| x + y \| \wedge \| x - y \| : x, y \in S_X \}. \]
A standard argument can be used to show 
\[J(X)  = \sup \{ \| x + y \| \wedge \| x - y \| : x, y \in B_X \}. \]
In that paper, it was demonstrated that $X$ is uniformly nonsquare if and only if  $J(X) < 2$ and the bounds on $J(X)$ are $\sqrt{2} \leq J(X) \leq 2.$

In [1] it was shown that if  $J(X) < 2$  then $X$ has the FPP.  The proof used the property that $RW(a, X) \leq J(X)$ and so is also indirect.

This paper gives a direct proof that $J(X) < 2$ implies $X$ has the FPP and thus shows a connection between the geometry of the unit sphere and the FPP.  The proof adapts the ideas and techniques that are found in [3].

\section{Results}

\begin{theorem}
Let $X$ be a Banach space where the dual unit ball, $B_{X^*},$ is weak*-sequentially compact.  If $X$ does not have the FPP then $J(X) = 2.$
\end{theorem}

\begin{proof} 
Let $X$ be a Banach space that does not have the FPP and where the dual unit ball, $B_{X^*},$ is weak*-sequentially compact.

Using results from Goebel [6], Karlovitz [7] and Lin [8] plus an excursion to $l_\infty(X)/c_0(X)$ and then back to $X$ it can be shown that the following can be assumed.
 
For $\epsilon > 0$ there is a sequence $(y_n)$ such that $ y_n \rightharpoonup y, \| y \| \leq \frac{1}{2}, \lim_{n \rightarrow \infty} \| y_n \| > 1 - \epsilon$ and 
$\lim_{n \rightarrow \infty} \lim_{m \rightarrow \infty} \| y_n - y_m \| \leq \frac{1}{2}.$ 

Without loss of generality, assume that $\| y_n \| > 1 - \epsilon$ for all $n$ and for large $m$ and $n, \| y_m - y_n \| \leq \frac{1}{2}+ \epsilon.$

For more background to these inequalities and to see some new ones, please refer to [3].

Using the weak lower semicontinuity of the norm, for all $m > 0, $
\[ \liminf_{n \rightarrow \infty} \| (y_m - y_n) + y \| \geq \| y_m \| > 1 - \epsilon. \]
So $\liminf_{m \rightarrow \infty} \liminf_{n \rightarrow \infty} \| (y_m - y_n) + y \| \geq 1 - \epsilon.$

By taking a subsequence, if necessary, assume $\| (y_m - y_n) + y \| \geq 1 - \epsilon \mbox{ for large } m, n.$

Consider $y_n^* \in S_{X^*}$ such that $y_n^*(y_n) = \| y_n \| \mbox{ for all } n.$  Because $B_{X^*}$ is w*-sequentially compact, we may assume $y_n^*\stackrel{*}{\rightharpoonup} y^*$ where  $\| y^* \| \leq 1.$

Again, for all $m > 0,$
\begin{align*}
\liminf_{n \rightarrow \infty} \| (y_m - y_n) - y \| & \geq \liminf_{n \rightarrow \infty} (-y_n^*)((y_m - y_n) - y)) \\
& = \liminf_{n \rightarrow \infty} y_n^*(y_n) - y^*(y_m - y) \\
& = \liminf_{n \rightarrow \infty} \| y_n \| - y^*(y_m - y).
\end{align*}
Since $y_m - y \rightharpoonup 0,$ we have for large $m, -\epsilon \leq y^*(y_m - y) \leq \epsilon.$

Therefore, for large $m$ and $n$,
\[ \| (y_m - y_n) - y \| \geq 1 - \epsilon - \epsilon = 1 - 2 \epsilon. \]
Because $\dfrac{ y_m - y_n}{ \frac{1}{2} + \epsilon}, \dfrac{ y}{\frac{1}{2} + \epsilon} \in B_{X},$

\begin{align*}
2 &\geq J(X) \geq \left \| \frac{y_m - y_n}{\frac{1}{2} + \epsilon} + \frac{y}{\frac{1}{2} + \epsilon} \right \| \wedge \left \| \frac{y_m - y_n}{\frac{1}{2} + \epsilon} -  \frac{y}{\frac{1}{2} + \epsilon}\right \|\\
& \geq \frac{1}{\frac{1}{2} + \epsilon} \Big ( \| (y_m - y_n) + y \| \wedge \| (y_m - y_n) - y \| \Big ) \\
& \geq  \frac{1}{\frac{1}{2} +  \epsilon} \Big ( (1 - \epsilon) \wedge (1 - 2\epsilon) \Big )\\
&= \frac{1 - 2\epsilon}{\frac{1}{2} + \epsilon}. 
\end{align*}

Taking $\epsilon \rightarrow 0, 2 \geq J(X) \geq 2.$  So $J(X) = 2.$.

\end{proof}

\begin{corollary}
If $X$ is uniformly nonsquare then $X$ has the FPP.
\end{corollary}

\begin{proof}
Any uniformly nonsquare Banach space is superreflexive and hence reflexive.  This means  $B_{X^*}$ is weak*-sequentially compact.

\end{proof}

\end{document}